\documentclass[12pt,twoside]{article}
\usepackage{a4}
\usepackage{amssymb,amsmath,amsthm,latexsym}
\usepackage{amsfonts}
\usepackage{amsfonts}
\usepackage{graphicx}
\newtheorem{theorem}{Theorem}[section]

\newtheorem{lemma} [theorem]{Lemma}

\newtheorem{proposition}[theorem]{Proposition}

\vspace{5cm}
\allowdisplaybreaks
\begin{document}
\label{'ubf'}  
\setcounter{page}{1}                                 
\markboth {\hspace*{-9mm} \centerline{\footnotesize \sc
   Algebraic Identities  }
                 }
                { \centerline                           {\footnotesize \sc  
      Tapas Chatterjee and Sonam Garg                                                } \hspace*{-9mm}              
               }
\vspace*{-2cm}

\begin{center}
{ {\Large \textbf { \sc  Algebraic identities among $q$- analogue of Euler double zeta values}}
\\
\medskip
\author{Chatterjee, Tapas \& Garg, Sonam} 
{\sc Tapas Chatterjee\footnote{Research of the first author is partly supported by the core research grant CRG/2023/000804 of the Science and Engineering Research Board of DST, Government of India.} and Sonam Garg\footnote{Research of the second author is supported by the University Grants Commission (UGC), India under File No.: 972/(CSIR-UGC NET JUNE 2018).} }\\
{\footnotesize Department of Mathematics, Indian Institute of Technology Ropar, Punjab, India.}\\
{\footnotesize e-mail: {\it {}$^1$tapasc@iitrpr.ac.in, {}$^2$2018maz0009@iitrpr.ac.in}}
}
\end{center}
\thispagestyle{empty}
\hrulefill
\begin{abstract}  
{\footnotesize  In 2003, Zudilin presented a $q$-analogue of Euler's identity for one of the variants of $q$-double zeta function. This article focuses on exploring identities related to another variant of $q$-double zeta function and its star variant. Using a $q$-analogue of the Nielsen Reflexion Formula for $q>1$, we investigate identities involving different versions of $q$-analogues of the Riemann zeta function and the double-zeta function.

Additionally, we analyze the behavior of $\zeta_q(s_1, s_2)$ as $s_1$ and $s_2$ approach to $0$ and compare these limits to those of the classical double-zeta function. Finally, we discuss the $q$-analogue of the Mordell-Tornheim $r$-ple zeta function and its relation with the $q$-double zeta function.  
}
 \end{abstract}
 \hrulefill



\section{Introduction}
The multiple zeta function, denoted by $\zeta(s_1,\hdots, s_r)$, is a mathematical function that is a natural generalization of the Riemann zeta function. It is defined as
\begin{align*}
\zeta(s_1,\hdots, s_r)= \sum_{n_1>n_2 >\hdots> n_r >0} \frac{1}{n_1^{s_1}\cdots n_k^{s_r}} = \sum_{n_1>n_2 \hdots > n_r >0} \prod_{i=1}^{r}\frac{1}{n_i^{s_i}}
\end{align*} 
where $s_1 > 1$ and $s_k\geq 1$ for $2 \leq k \leq r$ and converge when $\Re(s_1)+ \cdots + \Re(s_i)>i$, for all $i$. The multiple zeta functions, similar to the Riemann zeta function, are meromorphic and can be continued analytically in $\mathbb{C}^r$. These sums are known as multiple zeta values (MZVs) or Euler sums when $s_1,\hdots,s_r$ are all positive integers (with $s_1 > 1$). Here, $r$ is called the depth and $s_1+\cdots+s_r$ is called the weight of an MZV. The multiple zeta function has another variation known as multiple zeta star function, denoted as $\zeta^*(s_1,\ldots,s_r)$,  which is given as follows
\begin{align*}
\zeta^*(s_1,\hdots, s_r)= \sum_{n_1\geq n_2 \geq \hdots\geq n_r \geq 1} \frac{1}{n_1^{s_1}\cdots n_k^{s_r}} = \sum_{n_1 \geq n_2 \geq \hdots \geq n_r \geq 1} \prod_{i=1}^{r}\frac{1}{n_i^{s_i}},
\end{align*} 
where these sums are known as multiple zeta star values of depth $r$ and weight $n$. There is another well-established generalization of a function known as its $q$-analogue. In the domain of $q$-series, our primary object is to discover mathematical counterparts that behave as the original object when we have $q \to 1$. The $q$-analogues of a function are a type of generalization that depends on a parameter $q$. It reduces to the known classical equation as $q \to 1$ for any real number between $0 < q < 1$, or $q > 1$. Specifically, the $q$-analogue of a complex number $a$ is expressed as follows:
\begin{align*}
[a]_q = \frac{q^a -1}{q-1}, ~~q \neq 1.
\end{align*} For $q >1$, Kurokawa and Wakayama in \cite{KW} and Chatterjee and Garg in \cite{TSG} studied the following $q$- analogue of the Riemann zeta function
\begin{align} \label{E15}
\zeta_q(s)=\sum_{n=1}^{\infty}\frac{q^n}{[n]_q^s} ~~~~ \text{for}~~~~ \Re(s)>1.
\end{align} 
It should be noted that there are different ways to define $q$- analogues of the multiple zeta functions. One commonly studied $q$- analogue is proposed by Bradley in \cite{DMB}, which can be expressed as follows
\begin{align} \label{E16}
\zeta_q^B(s_1, s_2, \hdots, s_m) = \displaystyle \sum_{k_1> \hdots> k_m >0} \prod_{j=1}^m \frac{q^{(s_j-1)k_j}}{[k_j]^{s_j}_q},
\end{align}
where $s_1 > 1$ and $s_j\geq 1$ for $2 \leq j \leq m$. An alternative form of a $q$- analogue of the multiple zeta function was investigated by Ohno, Okuda, and Zudilin \cite{OOZ}. This particular $q$- analogue is defined as follows
\begin{align}
\overline{\mathfrak{z}}_q(s_1, \hdots, s_m)= \displaystyle \sum_{k_1> \hdots> k_m >0} \frac{q^{k_1}}{(1-q^{k_1})^{s_1} \cdots (1-q^{k_m})^{s_m} }. \label{QMZ2}
\end{align}
In addition to this, many other mathematicians, including Ebrahimi-Fard, Manchon and Singer in \cite{FMS}, Bachmann in \cite{HB}, Singer in \cite{JS} studied various variants of $q$-analogue of multiple zeta function, in particular double zeta values. In a recent article, Chatterjee and Garg \cite{TSG2} have explored another variant of a $q$-analogue of the multiple zeta function for $q>1$, denoted as 
\begin{align}
\zeta_q(s_1, s_2, \hdots, s_m) = \sum_{k_1> \hdots> k_m >0} \prod_{j=1}^m \frac{q^{k_j}}{[k_j]^{s_j}_q}. \label{20}
\end{align}
In particular, the series defining a $q$- double zeta function is 
\begin{align}
\zeta_q(s_1,s_2)= \sum_{k_1 > k_2 \geq 1} \frac{q^{k_1}q^{k_2}}{[k_1]^{s_1}_q[k_2]^{s_2}_q} = \sum_{k_1,k_2>0}\frac{q^{{k_1}+{k_2}}q^{k_2}}{[k_1+k_2]_q^{s_1}[k_2]_q^{s_2}}, \label{QMZ}
\end{align} 
where $s_1, s_2$ are complex numbers with $\Re(s_1) >1$ and $\Re(s_2) \geq 1$. They discussed the binomial expansion of the aforementioned $q$- double zeta function which is given by the expression
\begin{align}
\zeta_q(s_1,s_2)&=(q-1)^{s_1+s_2}\Bigg[\frac{1}{q^{s_1-1}-1}\Bigg\{\frac{1}{q^{s_1+s_2-2}-1}+\frac{s_2}{q^{s_1+s_2-1}-1}+\frac{s_2(s_2+1)}{2(q^{s_1+s_2}-1)}+\cdots\Bigg\} \nonumber\\
&\qquad +s_1\frac{1}{q^{s_1}-1}\Bigg\{\frac{1}{q^{s_1+s_2-1}-1}+\frac{s_2}{q^{s_1+s_2}-1}+\frac{s_2(s_2+1)}{2(q^{s_1+s_2+1}-1)}+\cdots\Bigg\}\nonumber\\
&\qquad \qquad+\frac{s_1(s_1+1)}{2}\frac{1}{q^{s_1+1}-1}\Bigg\{\frac{1}{q^{s_1+s_2}-1}+\frac{s_2}{q^{s_1+s_2+1}-1} +\frac{s_2(s_2+1)}{2(q^{s_1+s_2+2}-1)}\nonumber\\
&\qquad \qquad \qquad +\cdots\Bigg\} + \cdots \Bigg] \label{E1}.
\end{align} 
They also investigated a closed-form expression for a $q$- analogue of Euler’s constant of height $2$, denoted by $\gamma_{0,0} (q)$, the constant term in the Laurent series expansion of the $q$- analogue of the double zeta function given by (\ref{QMZ}) around $s_1 = 1$ and $s_2 = 1$. This article focuses on the identities of a $q$- double zeta values of the particular $q$- analogue defined in (\ref{QMZ}). Moreover, we discuss the identities involving the $q$- analogue of double zeta star function. Similar to the classical case, $q$- analogue of multiple zeta star function is given as
\begin{align*}
\zeta_q^*(s_1, s_2, \hdots, s_m) = \displaystyle \sum_{k_1 \geq \hdots \geq k_m \geq 1} \prod_{j=1}^m \frac{q^{k_j}}{[k_j]^{s_j}_q}.
\end{align*}
In particular, the series defining a $q$- analogue of double zeta star function is given by
\begin{align}
\zeta_q^*(s_1, s_2) = \sum_{k_1 \geq k_2 \geq 1} \frac{q^{k_1}q^{k_2}}{[k_1]^{s_1}_q[k_2]^{s_2}_q}.  \label{E19}
\end{align}
The study of identities among multiple zeta values has been an active area of research for several decades. The identities describe how multiple zeta values of a given weight and depth can be expressed in terms of multiple zeta values of lower weight and lower depths. In literature, various identities related to multiple zeta values are studied by different mathematicians. In 1775, Euler \cite{LE} proved the following identity $$\zeta(n)= \displaystyle\sum_{j=1}^{n-2}\zeta(n-j,j)$$ which holds for any integer $n \geq3$. In particular, he proved that $$\zeta(2,1) = \zeta(3).$$ In 2000, Hoffman and Ohno presented an identity that holds for an admissible sequence of positive integers $\textbf{s}=(s_1,s_2, \hdots, s_l)$ (with $s_1 > 1$), which is expressed as follows $$\displaystyle\sum_{k=1}^l \zeta(s_k+1, s_{k+1}, \hdots, s_l,s_1, \hdots, s_{k-1})= \sum_{\substack{k=1\\s_k \geq 2}}^l \sum_{j=0}^{s_k-2} \zeta(s_k-j, s_{k+1}, \hdots, s_l,s_1, \hdots, s_{k-1}, j+1).$$ Further, Gangl, Kaneko and Zagier in \cite{GKZ} proved the following identities 
\begin{align*}
\sum_{m=1}^{n-1}\zeta(2m, 2n-2m) &= \frac{3}{4}\zeta(2n), ~~~~\text{for}~ n>1\\
\sum_{m=1}^{n-1}\zeta(2m+1, 2n-2m-1) &= \frac{1}{4}\zeta(2n), ~~~~\text{for}~ n>1.
\end{align*}
In addition to these, various other identities for different weights have been studied in the literature, including
\begin{align*}
\zeta(2)\zeta(2) &= 2 \zeta(2,2) + \zeta(4)\\
2\zeta(2,2,1) + \zeta(2,1,2) + \zeta(4,1)& = \zeta(3,2) + \zeta (5)\\
\zeta(5,1) + \zeta(4,2) &= \zeta(4,1,1) + \zeta(3,2,1) + \zeta(2,3,1)\\
\zeta^*(4,1,2) &= \zeta(4,1,2) + \zeta(5,2) + \zeta(4,3) + \zeta(7)\\
\zeta(2,5,3) &= \zeta^*(2,5,3) - \zeta^*(7,3) - \zeta^*(2,8) + \zeta^*(10).
\end{align*}
Furthermore, for integers $s, s^{\prime} \geq 2$ the Nielsen Reflexion Formula given by
\begin{align}
\zeta(s) \zeta(s^{\prime}) = \zeta(s,s^{\prime}) + \zeta(s^{\prime},s) + \zeta(s + s^{\prime}), \label{E11}
\end{align}
is another well-known identity studied in the literature. Also, $q$- analogue the Nielsen Reflexion Formula for the variant given in (\ref{E16}) is given as
\begin{align*}
\zeta[s] \zeta[s^{\prime}] = \zeta[s,s^{\prime}] + \zeta[s^{\prime},s] + \zeta[s + s^{\prime}] + (1-q)\zeta[s + s^{\prime}-1].
\end{align*}
In this article, we discuss the following $q$-analogue of the Nielsen Reflexion Formula for $q>1$ for the variant given by the expression (\ref{20}) and (\ref{E19}), respectively,
\begin{align}
\zeta_q(s) \zeta_q(s^{\prime})& = \zeta_q(s,s^{\prime}) + \zeta_q(s^{\prime},s) + \zeta_q(s + s^{\prime}) + (q-1)\zeta_q(s + s^{\prime}-1) \nonumber\\ 
\zeta_q(s) \zeta_q(s^{\prime})& = \zeta_q^*(s,s^{\prime}) + \zeta_q^*(s^{\prime},s) - \zeta_q(s + s^{\prime}) - (q-1)\zeta_q(s + s^{\prime}-1), \label{E12}
\end{align}
where $s, s^{\prime} \geq 2$ and then use it to explore identities involving different versions of a $q$-analogue of the Riemann zeta function and the double-zeta function. The proof of the expression (\ref{E12}) can be easily seen from the following:
\begin{align*}
\zeta_q(s) \zeta_q(s^{\prime})&= \sum_{n=1}^{\infty}\frac{q^n}{[n]_q^s} \sum_{m=1}^{\infty}\frac{q^m}{[m]_q^{s^{\prime}}}\\
&= \sum_{n > m \geq 1} \frac{q^{n}q^{m}}{[n]^{s}_q[m]^{s^{\prime}}_q} + \sum_{m > n \geq 1} \frac{q^{m}q^{n}}{[m]^{s^{\prime}}_q[n]^{s}_q} + \sum_{n \geq 1} \frac{q^{2n}}{[n]^{s + s^{\prime}}_q} \\
& = \sum_{n > m \geq 1} \frac{q^{n}q^{m}}{[n]^{s}_q[m]^{s^{\prime}}_q} + \sum_{m > n \geq 1} \frac{q^{m}q^{n}}{[m]^{s^{\prime}}_q[n]^{s}_q} + \sum_{n \geq 1} \frac{q^{n}}{[n]^{s + s^{\prime}}_q} + (q-1)\sum_{n \geq 1} \frac{q^{n}}{[n]^{s + s^{\prime} -1}_q}\\
&= \zeta_q(s,s^{\prime}) + \zeta_q(s^{\prime},s) + \zeta_q(s + s^{\prime}) + (q-1)\zeta_q(s + s^{\prime}-1).
\end{align*}
Similarly, we will get the proof for the other expression in (\ref{E12}) related to the star variant.\\
Moreover, the limiting values of a $q$-analogue of the double zeta function are investigated as $s_1\to 0$ and $s_2\to 0$, and compared with the limiting values of the classical double zeta function.\\
The objective of this study is to broaden our understanding of this crucial area of mathematics by providing new insights. To achieve this goal, we present the following theorems.
\noindent
\begin{theorem}\label{P1}
Let $n_1$, $n_2$ be two integers and consider the $q$- double zeta function $\zeta_q(s_1,s_2)$ defined in (\ref{QMZ}). We define the following limits
\begin{align*}
\zeta_q(n_1,n_2)=\lim_{s_1\rightarrow n_1} \lim_{s_2\rightarrow n_2} \zeta_q(s_1,s_2)
\end{align*}
and
\begin{align*}
\zeta_q^R(n_1,n_2)=\lim_{s_2\rightarrow n_2} \lim_{s_1\rightarrow n_1} \zeta_q(s_1,s_2),
\end{align*}
whenever they exist. Then, we have
\begin{align*}
\lim_{q \rightarrow 1}\zeta_q(0,0) = \frac{5}{12} = \zeta^R(0,0) ~~~\text{and}~~~ \lim_{q \rightarrow 1}\zeta_q^R(0,0) = \frac{1}{3} = \zeta(0,0),
\end{align*}
where $$\zeta(n_1,n_2)=\lim_{s_1\rightarrow n_1} \lim_{s_2\rightarrow n_2} \zeta(s_1,s_2)$$
and
$$\zeta^R(n_1,n_2)=\lim_{s_2\rightarrow n_2} \lim_{s_1\rightarrow n_1} \zeta(s_1,s_2),$$
and $\zeta(s_1,s_2)$ is the classical double zeta function.
\end{theorem}

\noindent
In 2003, Zudilin \cite{WZ} presented a $q$-analogue of Euler's formula which is given as $$2\zeta_q(2,1) = \zeta_q(3),$$ where $$\zeta_q(2,1) = \displaystyle \sum_{n_1>n_2 \geq1} \frac{q^{n_1}}{(1-q^{n_1})^2(1-q^{n_2})}~~\text{and}~~\zeta_q(3) = \sum_{n=1}^{\infty}\frac{q^n(1+q^n)}{(1-q^n)^3}.$$
In this article, we investigate similar type of identities for the $q$- analogue given by (\ref{QMZ}) for higher weights. However, before that we introduce the following function for $q>1$ that is akin to the $q$- analogue defined by Ohno, Okuda, and Zudilin in expression (\ref{QMZ2})
\begin{align}
\zeta^{\circ}_q(s_1,s_2)= \sum_{n_1 > n_2 \geq 1}\frac{q^{n_1}}{[n_1]_q^{s_1}[n_2]_q^{s_2}} = \sum_{n_1,n_2>0}\frac{q^{{n_1}+{n_2}}}{[n_1+n_2]_q^{s_1}[n_2]_q^{s_2}}. \label{E23}
\end{align}
Now, consider the following binomial expansion of the expression (\ref{E23}):
\begin{align*}
\zeta^{\circ}_q(s_1,s_2)&=\sum_{n_1,n_2 \geq 1}\frac{q^{{n_1}+{n_2}}}{[n_1+n_2]_q^{s_1}[n_2]_q^{s_2}} =\sum_{n_1,n_2 \geq 1}\frac{q^{n_1+n_2}(q-1)^{s_1}}{(q^{n_1+n_2}-1)^{s_1}}\frac{(q-1)^{s_2}}{(q^{n_2}-1)^{s_2}}\\
&=(q-1)^{s_1+s_2}\sum_{n_1,n_2 \geq 1}q^{n_1+n_2}(q^{n_1+n_2}-1)^{-s_1}(q^{n_2}-1)^{-s_2}\\
&=(q-1)^{s_1+s_2}\sum_{n_2 \geq 1}(q^{n_2}-1)^{-s_2}\sum_{n_1 \geq 1}q^{n_1+n_2}(q^{n_1+n_2}-1)^{-s_1}\\
&=(q-1)^{s_1+s_2}\sum_{n_2 \geq 1}q^{-n_2s_2}(1-q^{-n_2})^{-s_2}\sum_{n_1 \geq 1}q^{n_1+n_2(1-s_1)}(1-q^{-(n_1+n_2)})^{-s_1}\\
&=(q-1)^{s_1+s_2}\sum_{n_2 \geq 1}q^{-n_2s_2}\sum_{k_2=0}^{\infty}\binom{-s_2}{k_2}(-1)^{k_2}q^{-n_2k_2}\\
&~~~~\sum_{n_1 \geq 1}q^{n_1+n_2(1-s_1)}\sum_{k_1=0}^{\infty}\binom{-s_1}{k_1}(-1)^{k_1}q^{-(n_1+n_2)k_1}\\
&=(q-1)^{s_1+s_2}\sum_{k_2=0}^{\infty}\binom{-s_2}{k_2}(-1)^{k_2}\sum_{n_2 \geq 1}q^{n_2(-s_2-k_2)}\\
&~~~~\sum_{k_1=0}^{\infty}\binom{-s_1}{k_1}(-1)^{k_1}\sum_{n_1 \geq 1}q^{n_1+n_2(1-s_1-k_1)}\\
&=(q-1)^{s_1+s_2}\sum_{k_2=0}^{\infty}\frac{s_2(s_2+1)\cdots(s_2+k_2-1)}{k_2!}\sum_{k_1=0}^{\infty}\frac{s_1(s_1+1)\cdots(s_1+k_1-1)}{k_1!}\\
&~~~~\sum_{n_2 \geq 1}q^{n_2(-s_2-k_2)}\sum_{n_1 \geq 1}q^{n_1+n_2(1-s_1-k_1)}\\
&=(q-1)^{s_1+s_2}\sum_{k_2=0}^{\infty}\frac{s_2(s_2+1)\cdots(s_2+k_2-1)}{k_2!}\sum_{k_1=0}^{\infty}\frac{s_1(s_1+1)\cdots(s_1+k_1-1)}{k_1!}\\
&~~~~\sum_{n_2 \geq 1}q^{-n_2(s_2+k_2+s_1+k_1-1)}\sum_{n_1 \geq 1}q^{-n_1(s_1+k_1-1)}\\
&=(q-1)^{s_1+s_2}\sum_{k_2=0}^{\infty}\frac{s_2(s_2+1)\cdots(s_2+k_2-1)}{k_2!}\sum_{k_1=0}^{\infty}\frac{s_1(s_1+1)\cdots(s_1+k_1-1)}{k_1!}\\
&~~~~\Bigg(\frac{1}{q^{(s_2+k_2+s_1+k_1-1)}-1}\Bigg)\Bigg(\frac{1}{q^{(s_1+k_1-1)}-1}\Bigg).
\end{align*}
Clearly, $\zeta^{\circ}_q(s_1,s_2)$ is variant of a $q$-analogue of double zeta function which is meromorphic for $s_1$, $s_2 \in \mathbb{C}$ with simple pole for $s_1 \in \big\{1 +i\frac{2\pi b}{\log q} \mid b \in \mathbb{Z}\big\}~ \cup ~\big\{ a + i\frac{2\pi  b}{\log q} \mid a,b \in \mathbb{Z}, a \leq 0, b \neq 0\big\}$ or $s_1 + s_2 \in \big \{ a + i \frac{2\pi b}{\log q}\mid a,b \in \mathbb{Z}, a \leq 0, b\neq 0 \big\} ~ \cup ~ \big\{1 +i\frac{2\pi b}{\log q} \mid b \in \mathbb{Z}\big\}$. Also, its star variant is given as
\begin{align}
\zeta^{\circ *}_q(s_1,s_2)=\sum_{n_1 \geq n_2 \geq 1}\frac{q^{n_1}}{[n_1]_q^{s_1}[n_2]_q^{s_2}}. \label{E24}
\end{align}
With these functions, we can now state our theorems and propositions as follows.
\begin{theorem}\label{P2}
The following identities hold
\begin{align*}
\zeta^{\circ}_q(3,1)& = \zeta_q(4) - \zeta_q(2,2) +(q-1)\zeta_q(3) = (\zeta_q(2))^2 - 3 \zeta_q(2,2).\\
\zeta^{\circ}_q(4,1)& = \zeta_q(5) - \zeta_q(2,3) - \zeta_q(3,2)+ (q-1)\zeta_q(4)\\
& \qquad= \zeta_q(2) \zeta_q(3) - 2 \zeta_q(2,3) - 2 \zeta_q(3,2).\\
\zeta_q^{\circ}(5,1)&= \zeta_q(6) - \zeta_q(3,3) -\zeta_q(4,2) - \zeta_q(2,4) + (q-1)\zeta_q(5)\\
&\qquad = (\zeta_q(3))^2 - 3 \zeta_q(3,3) -\zeta_q(4,2) - \zeta_q(2,4)\\
& \qquad \qquad = \zeta_q(2) \zeta_q(4) - \zeta_q(3,3) -2 \zeta_q(4,2) - 2 \zeta_q(2,4).\\
(\zeta_q(3))^2 - 2 \zeta_q(3,3)& = \zeta_q(2) \zeta_q(4) - \zeta_q(2,4) - \zeta_q(4,2).
\end{align*}
\end{theorem}
\noindent
\begin{theorem} \label{P3}
For $s \geq 3$, \textup{Theorem \ref{P2}} can be generalized as follows
\begin{align*}
\zeta_q^{\circ}(s,1) = \zeta_q(s+1)- \sum_{i=2}^{s-1}\zeta_q(s+1-i,i) + (q-1)\zeta_q(s).
\end{align*}
Further, depending on the parity of $s$, we have\\
\textbf{Case 1:} If $s$ is odd.
\begin{align}
\zeta_q^{\circ}(s,1) = \zeta_q(r) \zeta_q(r^{\prime}) -2 \zeta_q (r, r^{\prime})- 2 \zeta_q (r^{\prime}, r) - \sum_{i=2}^{s-1}\zeta_q(s+1-i,i), \label{E7}
\end{align}
where $r \geq 2$, $r^{\prime} \geq 2$ and $r + r^{\prime} = s+1$. So, there are $\Big(\frac{s-1}{2}\Big)$ possibilities to write $\zeta_q^{\circ}(s,1)$ in (\ref{E7}) when $s$ is odd.\\
\textbf{Case 2:} If $s$ is even. 
\begin{align}
\zeta_q^{\circ}(s,1) = \zeta_q(t) \zeta_q(t^{\prime}) -2 \zeta_q (t, t^{\prime})- 2 \zeta_q (t^{\prime}, t) - \sum_{\substack{i=2\\i \neq t, t^{\prime}}}^{s-1}\zeta_q(s+1-i,i), \label{E8}
\end{align}
where $t \geq 2$, $t^{\prime} \geq 2$ and $t + t^{\prime} = s+1$. So, there are $\Big(\frac{s-2}{2}\Big)$ possibilities to write $\zeta_q^{\circ}(s,1)$ in (\ref{E8}) when $s$ is even. 
\end{theorem}
\begin{proposition} \label{P5}
For $s \geq 3$, we have the following identities
\begin{align*}
\zeta_q^{\circ*}(s,1) = s\zeta_q(s+1)- \sum_{i=2}^{s-1}\zeta_q^*(s+1-i,i) + (s-1)(q-1)\zeta_q(s).
\end{align*}
Further, depending on the parity of $s$, we have\\
\textbf{Case 1:} If $s$ is odd.
\begin{align}
\zeta_q^{\circ *}(s,1)& = \zeta_q(r) \zeta_q(r^{\prime}) -2 \zeta_q^* (r, r^{\prime})- 2 \zeta_q^* (r^{\prime}, r) + (s+1)\zeta_q(s+1) + (q-1)s\zeta_q(s) \nonumber \\
& \qquad - \sum_{i=2}^{s-1}\zeta_q^*(s+1-i,i), \label{E21}
\end{align}
where $r \geq 2$, $r^{\prime} \geq 2$ and $r + r^{\prime} = s+1$. So, there are $\Big(\frac{s-1}{2}\Big)$ possibilities to write $\zeta_q^{\circ *}(s,1)$ in (\ref{E21}) when $s$ is odd.\\
\textbf{Case 2:} If $s$ is even. 
\begin{align}
\zeta_q^{\circ *}(s,1)& = \zeta_q(t) \zeta_q(t^{\prime}) -2 \zeta_q^* (t, t^{\prime})- 2 \zeta_q^* (t^{\prime}, t) + (s+1)\zeta_q(s+1) + (q-1)s\zeta_q(s) \nonumber\\
& \qquad - \sum_{\substack{i=2\\i \neq t, t^{\prime}}}^{s-1}\zeta_q(s+1-i,i), \label{E22}
\end{align}
where $t \geq 2$, $t^{\prime} \geq 2$ and $t + t^{\prime} = s+1$. So, there are $\Big(\frac{s-2}{2}\Big)$ possibilities to write $\zeta_q^{\circ *}(s,1)$ in (\ref{E22}) when $s$ is even. 
\end{proposition}
\noindent
Further, Tornheim in \cite{LT} considered the following double series
\begin{align} \label{E13}
\sum_{m \geq 1} \sum_{n \geq 1} \frac{1}{m^{s_1} n^{s_2} (m+n)^{s_3}},
\end{align} 
where $s_1$, $s_2$ and $s_3$ are nonnegative integers satisfying $s_1 + s_3 >1$, $s_2 + s_3 >1$ and $s_1 + s_2 + s_3 >2$, which he referred to as the harmonic double series. The special case of (\ref{E13}) was studied by Mordell in \cite{LJM}, where he considered $s_1 = s_2 = s_3$ and also investigated the following multiple sum
\begin{align*}
\sum_{m_1 \geq 1} \cdots  \sum_{m_r \geq 1} \frac{1}{m_1 \cdots m_r (m_1 + \cdots + m_r + a)},
\end{align*}
where $a > -r$. In \cite{KM, KM1}, Matsumoto referred to the series given in (\ref{E13}) as the Mordell-Tornheim zeta function and introduced the following multi-variable version of (\ref{E13})
\begin{align} \label{E14}
\zeta_{MT}(s_1,s_2, \ldots, s_r; s_{r+1}) = \sum_{m_1 \geq 1} \cdots  \sum_{m_r \geq 1} \frac{1}{m_1^{s_1} \cdots m_r^{s_r}(m_1 + \cdots + m_r)^{s_{r+1}}},
\end{align}
where $s_1, \ldots, s_{r+1} \in \mathbb{C}$ and the series converges absolutely when $\Re(s_j) > 1 ~(1 \leq j \leq r)$ and $\Re(s_{r+1}) >0$. Matsumoto referred to it as  Mordell-Tornheim $r$-ple zeta function and also proved the meromorphic continuation of the series to the whole $\mathbb{C}^{r+1}$ space, exhibiting singularities solely on subsets of $\mathbb{C}^{r+1}$ defined by one of the following equations
\begin{align*}
&s_j + s_{r+1} = 1-l~~(1 \leq j \leq r, l \in \mathbb{N}_0),\\
&s_{j_1}+ s_{j_2} + s_{r+1} = 2-l~~(1 \leq s_{j_1} < s_{j_2} \leq r, l \in \mathbb{N}_0),\\
&\qquad\qquad\cdots\cdots\\
&s_{j_1}+ \cdots s_{j_{r-1}} + s_{r+1} = r-1-l~~(1 \leq s_{j_1} < \cdots < s_{j_{r-1}} \leq r, l \in \mathbb{N}_0),\\
&s_1 + \cdots + s_r + s_{r+1} = r,
\end{align*}
where $\mathbb{N}_0$ denotes the set of non-negative integers. In this article, we introduce the following $q$-analogue of the expression (\ref{E14})
\begin{align*}
\zeta_{MT,q}(s_1,s_2, \ldots, s_r; s_{r+1}) = \sum_{m_1 \geq 1} \cdots  \sum_{m_r \geq 1} \frac{q^{m_1} \cdots q^{m_r} q^{m_1 + \cdots + m_r}}{[m_1]_q^{s_1} \cdots [m_r]_q^{s_r}[m_1 + \cdots + m_r]_q^{s_{r+1}}},
\end{align*}
where $s_1, \ldots, s_{r+1} \in \mathbb{C}$ and $q>1$. We call it $q$- Mordell-Tornheim $r$-ple zeta function. For $r=2$, we have 
\begin{align*}
\zeta_{MT,q}(s_1,s_2; s_3) = \sum_{m_1 \geq 1}\sum_{m_2 \geq 1} \frac{q^{m_1} q^{m_2} q^{m_1 + m_2}}{[m_1]_q^{s_1} [m_2]_q^{s_2}[m_1 + m_2]_q^{s_3}}.
\end{align*}

\begin{theorem}\label{P4}
Let $s \geq 2$ and $r \geq 3$ be any two integers, then we have the following identity
\begin{align*}
\zeta_q(s,r) = \zeta_q(s)(\zeta_q(r) + (q-1)\zeta_q(r-1)) - \sum_{j=0}^{s-1}\zeta_{MT,q}(r, j+1; s-j),
\end{align*} 
where $\zeta_q(s)$ is the $q$-analogue of the Riemann zeta function given by the expression (\ref{E15}).
\end{theorem}

\section{\bf Notations and Preliminaries}
In this section, we introduce some notations and definitions that are relevant to $q$-series. We also discuss some important findings that are crucial for the subsequent sections. The theory of $q$-series aims to find $q$-analogues of mathematical objects. The goal is to generalize or extend the properties of the original mathematical object to a more general setting. In this article, we assume that $q > 1$.\\
The $q$-analogue of a complex number $a$ is expressed as follows
\begin{align*}
[a]_q = \frac{q^a -1}{q-1}, ~~q \neq 1.
\end{align*} 
The $q$- factorial is defined as
\begin{align*}
[n]_q!&= [1]_q\cdot[2]_q \cdots [n-1]_q\cdot [n]_q = \tfrac{q - 1}{q - 1}\cdot\tfrac{q^2 - 1}{q - 1} \cdots \tfrac{q^n - 1}{q - 1}\\
&\quad = 1 \cdot (1+q) \cdots (1+q+q^2+ \cdots +q^{(n-1)})
\end{align*}
and the $q$- shifted factorial of $a$ is given by
\begin{align*}
(a;q)_0&=1, ~~~~~ (a;q)_n = \prod_{m=0}^{n-1} (1-aq^m), ~~~~~ n\geq1,\\
(a;q)_{\infty}& =\lim_{n\rightarrow\infty}(a;q)_n = \displaystyle\prod_{n=0}^{\infty} (1-aq^n).
\end{align*}
Furthermore, the subsequent partial fraction expression plays a crucial role in proving Theorems \ref{P2} and \ref{P3}
\begin{align}
\frac{1}{(1-u)(1-uv)^s}=\frac{1}{(1-u)(1-v)^s} - \sum_{i=0}^{s-1}\frac{v}{(1-v)^{i+1}(1-uv)^{s-i}}, \label{E6}
\end{align}
where $u$, $v \in \mathbb{R}$.
\section{\bf Proofs of the main theorems.}

\begin{proof}[Proof of Theorem \ref{P1}]
From the expression (\ref{E1}), we have
\begin{align*}
\zeta_q(s_1,s_2)&=(q-1)^{s_1+s_2}\Bigg[\frac{1}{q^{s_1-1}-1}\Bigg\{\frac{1}{q^{s_1+s_2-2}-1}+\frac{s_2}{q^{s_1+s_2-1}-1}+\frac{s_2(s_2+1)}{2(q^{s_1+s_2}-1)}+\cdots\Bigg\}\\
&\qquad +s_1\frac{1}{q^{s_1}-1}\Bigg\{\frac{1}{q^{s_1+s_2-1}-1}+\frac{s_2}{q^{s_1+s_2}-1}+\frac{s_2(s_2+1)}{2(q^{s_1+s_2+1}-1)}+\cdots\Bigg\}\\
&\qquad \qquad +\frac{s_1(s_1+1)}{2}\frac{1}{q^{s_1+1}-1}\Bigg\{\frac{1}{q^{s_1+s_2}-1}+\frac{s_2}{q^{s_1+s_2+1}-1} +\frac{s_2(s_2+1)}{2(q^{s_1+s_2+2}-1)}\\
&\qquad \qquad \qquad +\cdots\Bigg\} + \cdots \Bigg].
\end{align*}
Clearly, for $n_1 =0$ and $n_2 =0$ we have
\begin{align*}
\zeta_q(0,0)&= \lim_{s_1 \rightarrow 0, s_2 \rightarrow 0} \zeta_q(s_1,s_2)\\
& \quad = \frac{1}{(q^{-1}-1)(q^{-2}-1)} + \frac{1}{(q^{-1}-1) \log q} + \frac{1}{2(q-1) \log q}\\
\zeta_q^R(0,0) &= \lim_{s_2 \rightarrow 0, s_1 \rightarrow 0} \zeta_q^R(s_1,s_2)\\
& \quad = \frac{1}{(q^{-1}-1)(q^{-2}-1)} + \frac{3}{2(q^{-1}-1) \log q} + \frac{1}{ \log^2 q}.
\end{align*}
Hence, using the following asymptotic formulas
\begin{align*}
\frac{1}{x+2} &= \frac{1}{2} -\frac{x}{4} + \frac{x^2}{8} + O[x^3] ~~~(x \rightarrow 0)\\
\frac{1}{\log(1+x)}& = \frac{1}{x} +\frac{1}{2} - \frac{x}{12} + \frac{x^2}{24} + O[x^3] ~~~(x \rightarrow 0)\\
\frac{1}{\log^2(1+x)}& = \frac{1}{x^2} +\frac{1}{x} + \frac{1}{12}+ 0\cdot  x + O[x^2] ~~~(x \rightarrow 0),
\end{align*}
we easily determine that 
\begin{align*}
\lim_{q \rightarrow 1}\zeta_q(0,0) = \frac{5}{12} = \zeta^R(0,0) ~~~\text{and}~~~ \lim_{q \rightarrow 1}\zeta_q^R(0,0) = \frac{1}{3} = \zeta(0,0).
\end{align*}
\end{proof}

\begin{proof}[Proof of Theorem \ref{P2}]
We employ the partial fraction method given in the expression (\ref{E6}).
For $s=3$, multiply the identity (\ref{E6}) by $uv$, set $u=q^m$ and $v=q^n$ and sum over all positive integers $m$ and $n$. As a consequence, we obtain an equality with the double sum on the left-hand side as
\begin{align*}
\sum_{m=1}^{\infty}\sum_{n=1}^{\infty}\frac{q^{n+m}}{(1-q^m)(1-q^{n+m})^3}= \sum_{n=1}^{\infty}\sum_{m=1}^{\infty}\frac{q^{n+m}}{(1-q^n)(1-q^{n+m})^3}
\end{align*}
and the double sum on the right-hand side as
\begin{align}
&\sum_{n=1}^{\infty}\sum_{m=1}^{\infty}\Bigg(\frac{q^{n+m}}{(1-q^n)^3(1-q^m)}- \frac{q^{2n+m}}{(1-q^n)(1-q^{n+m})^3} - \frac{q^{2n+m}}{(1-q^n)^2(1-q^{n+m})^2} \nonumber\\
&\qquad - \frac{q^{2n+m}}{(1-q^n)^3(1-q^{n+m})} \Bigg) \nonumber\\
&\quad =\sum_{n=1}^{\infty} \frac{q^n}{(1-q^n)^3}\sum_{m=1}^{\infty}\Bigg(\frac{q^m}{(1-q^m)} - \frac{q^{m+n}}{(1-q^{m+n})}\Bigg) - \sum_{n=1}^{\infty}\sum_{m=1}^{\infty} \frac{q^{2n+m}}{(1-q^n)^2(1-q^{n+m})^2} \nonumber \\
&\quad \qquad -\sum_{n=1}^{\infty}\sum_{m=1}^{\infty} \frac{q^{2n+m}}{(1-q^n)(1-q^{n+m})^3}. \label{E9}
\end{align}
Taking the last double sum of expression (\ref{E9}) to the left-hand side and multiplying both the sides by $(1-q)^4$, we have
\begin{align*}
&(1-q)^4\sum_{n=1}^{\infty}\sum_{m=1}^{\infty}\frac{q^{n+m}+ q^{2n+m} }{(1-q^n)(1-q^{n+m})^3}\\
&\quad =(1-q)^4\Bigg(\sum_{n=1}^{\infty}\frac{q^n}{(1-q^n)^3}\sum_{m=1}^{n}\frac{q^m}{(1-q^m)} \Bigg) - \zeta_q(2,2)\\
&\quad = (1-q)^4\Bigg(\sum_{n=1}^{\infty}\frac{q^{2n}}{(1-q^n)^4} + \sum_{n=1}^{\infty}\sum_{m=1}^{n-1}\frac{q^{n+m}}{(1-q^n)^3(1-q^m)} \Bigg) - \zeta_q(2,2).
\end{align*}
Using (\ref{E12}), we have 
\begin{align*}
\sum_{n=1}^{\infty}\frac{q^{2n}}{[n]_q^4}& = \sum_{n=1}^{\infty}\frac{q^{n}}{[n]_q^4} + (q-1)\sum_{n=1}^{\infty}\frac{q^{n}}{[n]_q^3}\\ 
&\qquad = \zeta_q(4) +(q-1)\zeta_q(3)\\
&\qquad \qquad =\zeta_q(2) \zeta_q(2) - \zeta_q(2,2) - \zeta_q(2,2).
\end{align*}
Thus,
\begin{align*}
\sum_{n=1}^{\infty}\sum_{m=1}^{\infty}\frac{q^{n+m}+ q^{2n+m} }{[n]_q[n+m]_q^3}& = \zeta_q(4) - \zeta_q(2,2) + (q-1)\zeta_q(3) + \sum_{n > m \geq 1}\frac{q^{n+m}}{[n]_q^3[m]_q},\\
&\qquad = \zeta_q(2)\zeta_q(2) - 3\zeta_q(2,2) + \sum_{n > m \geq 1}\frac{q^{n+m}}{[n]_q^3[m]_q}.
\end{align*}
Changing the variables $m+n = t$ on the left-hand side, we get
\begin{align*}
\sum_{n=1}^{\infty}\sum_{m=1}^{\infty}\frac{q^{n+m}+ q^{2n+m} }{[n]_q[n+m]_q^3}&= \sum_{n=1}^{\infty}\sum_{t=n+1}^{\infty}\frac{q^{t}+ q^{n+t} }{[n]_q[t]_q^3}\\
&= \sum_{t>n \geq 1} \frac{q^{t}+ q^{n+t} }{[n]_q[t]_q^3}. 
\end{align*} 
Finally, setting $n=n_1$, $m=n_2$ on the right-hand side and $t=n_1$ and $n=n_2$ on the left-hand side, we get
\begin{align*}
\sum_{n_1> n_2 \geq 1} \frac{q^{n_1} }{[n_1]_q^3[n_2]_q} &= \zeta_q(4) - \zeta_q(2,2) + (q-1)\zeta_q(3)\\
& \qquad = \zeta_q(2)\zeta_q(2) - 3\zeta_q(2,2).
\end{align*}
Thus,
\begin{align*}
\zeta^{\circ}_q(3,1) = \zeta_q(4) - \zeta_q(2,2) + (q-1)\zeta_q(3) =  (\zeta_q(2))^2 - 3\zeta_q(2,2).
\end{align*}
This completes the proof of the first identity. Working on similar lines, we get second and third identities. Hence, the fourth also.
\end{proof}
 
\begin{proof}[Proof of Theorem \ref{P3}] For $s \geq 3$, multiply both the sides of expression (\ref{E6}) by $uv$ and then set $u=q^m$ and $v=q^n$. By summing over all positive integers $m$ and $n$, the double sum on the left-hand side is given as
\begin{align*}
\sum_{m=1}^{\infty}\sum_{n=1}^{\infty}\frac{q^{n+m}}{(1-q^m)(1-q^{n+m})^s}= \sum_{n=1}^{\infty}\sum_{m=1}^{\infty}\frac{q^{n+m}}{(1-q^n)(1-q^{n+m})^s}
\end{align*}
and the double sum on the right-hand side is given as 
\begin{align}
&\sum_{n=1}^{\infty}\sum_{m=1}^{\infty}\Bigg(\frac{q^{n+m}}{(1-q^n)^s(1-q^m)}- \frac{q^{2n+m}}{(1-q^n)(1-q^{n+m})^s} - \frac{q^{2n+m}}{(1-q^n)^2(1-q^{n+m})^{s-1}} \nonumber \\
&\qquad- \frac{q^{2n+m}}{(1-q^n)^3(1-q^{n+m})^{s-2}}- \cdots - \frac{q^{2n+m}}{(1-q^n)^s(1-q^{n+m})} \Bigg) \nonumber \\
&\quad =\sum_{n=1}^{\infty} \frac{q^n}{(1-q^n)^s}\sum_{m=1}^{\infty}\Bigg(\frac{q^m}{(1-q^m)} - \frac{q^{m+n}}{(1-q^{m+n})}\Bigg) - \sum_{n=1}^{\infty}\sum_{m=1}^{\infty} \frac{q^{2n+m}}{(1-q^n)^2(1-q^{n+m})^{s-1}} \nonumber \\
&\quad \qquad -\sum_{n=1}^{\infty}\sum_{m=1}^{\infty}\frac{q^{2n+m}}{(1-q^n)^3(1-q^{n+m})^{s-2}} - \cdots -\sum_{n=1}^{\infty}\sum_{m=1}^{\infty}\frac{q^{2n+m}}{(1-q^n)^{s-1}(1-q^{n+m})^{2}} \nonumber\\
&\quad \qquad \qquad -\sum_{n=1}^{\infty}\sum_{m=1}^{\infty} \frac{q^{2n+m}}{(1-q^n)(1-q^{n+m})^s}. \label{E10}
\end{align}
Taking the last double sum of expression (\ref{E10}) to the left-hand side and multiplying both the sides by $(1-q)^{s+1}$, we get
\begin{align*}
&(1-q)^{s+1}\Bigg(\sum_{n=1}^{\infty}\sum_{m=1}^{\infty}\frac{q^{n+m}+ q^{2n+m} }{(1-q^n)(1-q^{n+m})^s}\Bigg)\\
&\quad =(1-q)^{s+1}\Bigg(\sum_{n=1}^{\infty}\frac{q^n}{(1-q^n)^s}\sum_{m=1}^{n}\frac{q^m}{(1-q^m)}\Bigg) - \zeta_q(s-1,2) - \cdots - \zeta_q(2,s-1)\\
&\quad = (1-q)^{s+1}\Bigg(\sum_{n=1}^{\infty}\frac{q^{2n}}{(1-q^n)^{s+1}} + \sum_{n=1}^{\infty}\sum_{m=1}^{n-1}\frac{q^{n+m}}{(1-q^n)^s(1-q^m)}\Bigg) - \zeta_q(s-1,2) \\
&\quad \qquad - \zeta_q(s-2,3)- \cdots - \zeta_q(2,s-1).
\end{align*}
From (\ref{E12}), we obtain
\begin{align*}
\sum_{n=1}^{\infty}\frac{q^{2n}}{[n]_q^{s+1}}& = \sum_{n=1}^{\infty}\frac{q^{n}}{[n]_q^{s+1}} + (q-1) \sum_{n=1}^{\infty}\frac{q^{n}}{[n]_q^{s}}\\
&\qquad = \zeta_q(s+1) + (q-1)\zeta_q(s)\\
& \qquad\qquad = \zeta_q(r)\zeta_q(r^{\prime}) - \zeta_q(r, r^{\prime}) -\zeta_q(r^{\prime},r),
\end{align*}
where $r \geq 2$, $r^{\prime} \geq 2$ and $r + r^{\prime} = s+1$.
So, working on similar lines as in Theorem \ref{P2}, we get
\begin{align*}
\zeta_q^{\circ}(s,1) = \zeta_q(s+1) - \sum_{i=2}^{s-1}\zeta_q(s+1-i,i) + (q-1)\zeta_q(s).
\end{align*}
Now, depending on whether $s$ is odd or even, we have the following two cases.\\
\textbf{Case 1:} If $s$ is odd, this implies $s+1$ is even, so for $r = r^{\prime} = \frac{s+1}{2}$, $$\sum_{n=1}^{\infty}\frac{q^{2n}}{[n]_q^{s+1}} = \Big(\zeta_q\Big(\frac{s+1}{2}\Big)\Big)^2 - 2 \zeta_q\Big(\frac{s+1}{2}, \frac{s+1}{2}\Big).$$
So, doing similar kind of calculations as in Theorem \ref{P2}, we get
\begin{align*}
\zeta_q^{\circ}(s,1) = \Big(\zeta_q\Big(\frac{s+1}{2}\Big)\Big)^2 - 3 \zeta_q \Big(\frac{s+1}{2},\frac{s+1}{2} \Big) - \sum_{\substack{i=2\\i \neq \frac{s+1}{2}}}^{s-1}\zeta_q(s+1-i,i).
\end{align*}
When $r \neq r^{\prime}$, the number of possible pairs $(r, r^\prime)$ that satisfy $r, r^\prime \geq 2$ and $r + r^\prime = s+1$ (where the order of addends does not matter) is given by $\frac{(s-3)}{2}$. So, for each such pair, we get
\begin{align*}
\zeta_q^{\circ}(s,1) = \zeta_q(r) \zeta_q(r^{\prime}) -2 \zeta_q (r, r^{\prime})- 2 \zeta_q (r^{\prime}, r) - \sum_{\substack{i=2\\i \neq r, r^{\prime}}}^{s-1}\zeta_q(s+1-i,i).
\end{align*}
\textbf{Case 2:} If $s$ is even, this implies $s+1$ is odd, then the number of possible pairs $(t, t^\prime)$ that satisfy $t, t^\prime \geq 2$ and $t + t^\prime = s+1$ are $\frac{(s-2)}{2}$.\\
Hence, the result follows.
\end{proof}
\begin{proof}[Proof of Proposition \ref{P5}] First, note that from (\ref{QMZ}) and (\ref{E19}), we have
\begin{align*}
\zeta_q^*(s^{\prime},s) = \zeta_q(s^{\prime},s) + \zeta_q(s^{\prime}+s) + (q-1)\zeta_q(s^{\prime}+s-1)
\end{align*}
and from (\ref{E23}) and (\ref{E24}), we have
\begin{align*}
\zeta_q^{\circ *}(s^{\prime},s) = \zeta_q^{\circ}(s^{\prime},s) + \zeta_q(s^{\prime}+s).
\end{align*}
Then, using the above two observations and Theorem \ref{P3}, we get the desired results.
\end{proof}
\noindent
Before giving a proof of Theorem \ref{P4}, we give a lemma which plays a pivotal role in the proof and it is a generalization of the partial fraction given by the expression (\ref{E6}). Here is the statement of the lemma.
\begin{lemma} \label{L2}
Let $s,r \geq 1$ be two integers, then 
\begin{align} \label{E17}
\frac{1}{(1-u)^r(1-uv)^s}=\frac{1}{(1-u)^r(1-v)^s} - \sum_{i=0}^{s-1}\frac{v(1-u)^{1-r}}{(1-v)^{i+1}(1-uv)^{s-i}},
\end{align}
where $u,v \in \mathbb{R}$.
\end{lemma}
\begin{proof}[Proof of Lemma \ref{L2}]
To prove this identity, it is enough to note that the second term on the right-hand side can be expressed as a geometric series with a common ratio of $\Big(\frac{1-uv}{1-v}\Big)$.
\end{proof}
\noindent
Note that for $r=1$, it gives back the identity given in the expression (\ref{E6}).
\begin{proof}[Proof of Theorem \ref{P4}]
Multiply the identity (\ref{E17}) by $u^2v$, set $u=q^m$ and $v=q^n$ and sum over all positive integers $m$ and $n$. As a consequence, we obtain an equality with the double sum on the left-hand side as
\begin{align*}
\sum_{m=1}^{\infty}\sum_{n=1}^{\infty}\frac{q^m q^{n+m}}{(1-q^m)^r(1-q^{n+m})^s}= \sum_{n=1}^{\infty}\sum_{m=1}^{\infty}\frac{q^n q^{n+m}}{(1-q^n)^r(1-q^{n+m})^s} 
\end{align*}
and the double sum on the right-hand side as
\begin{align*}
&\sum_{n=1}^{\infty}\sum_{m=1}^{\infty}\Bigg(\frac{q^m q^{n+m}}{(1-q^n)^s(1-q^m)^r}- \frac{q^n q^m q^{n+m}}{(1-q^m)^{r-1}(1-q^n)(1-q^{n+m})^s} \\
&\qquad-\cdots - \frac{q^n q^m q^{n+m}}{(1-q^m)^{r-1}(1-q^n)^s(1-q^{n+m})}\Bigg)\\
& \quad =\sum_{n=1}^{\infty} \frac{q^n}{(1-q^n)^s} \sum_{m=1}^{\infty}\frac{q^{2m}}{(1-q^m)^r} - \sum_{n=1}^{\infty}\sum_{m=1}^{\infty}\Bigg(\frac{q^n q^m q^{n+m}}{(1-q^m)^{r-1}(1-q^n)(1-q^{n+m})^s}\\
&\quad \qquad +\cdots + \frac{q^n q^m q^{n+m}}{(1-q^m)^{r-1}(1-q^n)^s(1-q^{n+m})}\Bigg)
\end{align*}
Multiplying both the sides by $(1-q)^{s+r}$ and using (\ref{E12}), we get the desired result.
\end{proof}

\label{'ubl'}  
\end{document}